\documentclass[reqno,a4paper]{amsart}
\usepackage{amssymb,setspace}
\usepackage{ifpdf}
\ifpdf
 \usepackage[hyperindex,pagebackref]{hyperref}
\else
 \expandafter\ifx\csname dvipdfm\endcsname\relax
 \usepackage[hypertex,hyperindex,pagebackref]{hyperref}
 \else
 \usepackage[dvipdfm,hyperindex,pagebackref]{hyperref}
 \fi
\fi
\allowdisplaybreaks[4]
\numberwithin{equation}{section}
\theoremstyle{plain}
\newtheorem{thm}{Theorem}[section]
\newtheorem{lem}{Lemma}[section]
\theoremstyle{remark}
\newtheorem{rem}{Remark}[section]
\DeclareMathOperator{\td}{d}

\begin{document}

\title[Approximation formulas and inequalities for Wallis ratio]
{Some best approximation formulas and inequalities for Wallis ratio}

\author[F. Qi]{Feng Qi}
\address[Qi]{College of Mathematics, Inner Mongolia University for Nationalities, Tongliao City, Inner Mongolia Autonomous Region, 028043, China; Institute of Mathematics, Henan Polytechnic University, Jiaozuo City, Henan Province, 454010, China}
\email{\href{mailto: F. Qi <qifeng618@gmail.com>}{qifeng618@gmail.com}, \href{mailto: F. Qi <qifeng618@hotmail.com>}{qifeng618@hotmail.com}, \href{mailto: F. Qi <qifeng618@qq.com>}{qifeng618@qq.com}}
\urladdr{\url{http://qifeng618.wordpress.com}}

\author[C. Mortici]{Cristinel Mortici}
\address[Mortici]{Department of Mathematics, Valahia University of T\^argovi\c{s}te, Bd. Unirii 18, 130082 T\^argovi\c{s}te, Romania}
\email{\href{mailto: C. Mortici <cmortici@valahia.ro>}{cmortici@valahia.ro}}
\urladdr{\url{http://www.cristinelmortici.ro}}

\begin{abstract}
In the paper, the authors establish some best approximation formulas and inequalities for Wallis ratio. These formulas and inequalities improve an approximation formula and a double inequality for Wallis ratio recently presented in ``S. Guo, J.-G. Xu, and F. Qi, \textit{Some exact constants for the approximation of the quantity in the Wallis' formula}, J. Inequal. Appl. 2013, \textbf{2013}:67, 7~pages; Available online at \url{http://dx.doi.org/10.1186/1029-242X-2013-67}''.
\end{abstract}



\thanks{This paper was typeset using \AmS-\LaTeX}

\maketitle

\section{Introduction}

Wallis ratio is defined as
\begin{equation*}
W_{n}=\frac{(2n-1)!!}{(2n)!!}=\frac{1}{\sqrt{\pi}\,}
\frac{\Gamma\bigl(n+\frac{1}{2}\bigr)}{\Gamma(n+1)},
\end{equation*}
where $\Gamma$ is the classical Euler gamma function which may be defined by
\begin{equation}\label{gamma-dfn}
\Gamma(z)=\int^\infty_0u^{z-1} e^{-u}\td u, \quad \Re(z)>0.
\end{equation}
The study and applications of $W_{n}$ have a long history, a large amount of literature, and a lot of new results. For detailed information, please refer to the papers~\cite{chenwallis, notes-best-simple-open-jkms.tex, mia-qi-cui-xu-99, Bukac-Sevli-Gamma.tex}, related texts in the survey articles~\cite{bounds-two-gammas.tex, Wendel2Elezovic.tex-JIA, Wendel-Gautschi-type-ineq-Banach.tex} and references cited therein. Recently, Guo, Xu, and Qi proved in~\cite{g} that the double inequality
\begin{equation}\label{u}
\sqrt{\frac{e}{\pi}}\,\biggl(1-\frac{1}{2n}\biggr)^{n}\frac{\sqrt{n-1}\,}{n}
<W_{n}\le \frac{4}{3}\biggl(1-\frac{1}{2n}\biggr)^{n}\frac{\sqrt{n-1}\,}{n}
\end{equation}
for $n\ge2$ is valid and sharp in the sense that the constants $\sqrt{\frac{e}{\pi}}\,$ and $\frac{4}{3}$ in~\eqref{u} are best possible. They also proposed in~\cite{g} the approximation formula
\begin{equation}\label{w}
W_{n}\sim \chi _{n}:=\sqrt{\frac{e}{\pi}}\,\biggl(1-\frac{1}{2n}\biggr)^{n}
\frac{\sqrt{n-1}\,}{n},\quad n\to\infty.
\end{equation}
\par
The sharpness of the double inequality~\eqref{u} was proved in~\cite{g} basing on the variation of a function which decreases on $[2,\infty)$ from $\frac43$ to $\sqrt{\frac{e}\pi}\,$. As a consequence, the right-hand side of~\eqref{u} becomes weak for large values of $n$. Moreover, if we are interested to estimating $W_{n}$ when $n$ approaches infinity, then the constant $\sqrt{\frac{e}\pi}\,$ should be chosen and inequalities using $\sqrt{\frac{e}\pi}\,$ are welcome.
\par
The aim of this paper is to improve the double inequality~\eqref{u} and the approximation formula~\eqref{w}.

\section{A lemma}

For improving the double inequality~\eqref{u} and the approximation formula~\eqref{w}, we need the following lemma.

\begin{lem}[{\cite[Lemma~1.1]{m1}}]\label{m1-Lemma1.1}
If the sequence $\{\omega_n:n\in\mathbb{N}\}$ converges to $0$ and
\begin{equation}
\lim_{n\to\infty}n^k(\omega_n-\omega_{n+1})=\ell\in\mathbb{R}
\end{equation}
for $k>1$, then
\begin{equation}
\lim_{n\to\infty}n^{k-1}\omega_n=\frac{\ell}{k-1}.
\end{equation}
\end{lem}

\begin{rem}
Lemma~\ref{m1-Lemma1.1} was first established in~\cite{m2} and has been effectively applied in many papers such as~\cite{Chen-Srivatava-AML-12, Furdui-ITSF-13, Laforgia-Natalini-JMAA-12, Lin-Long-JIA-13, Lu-Wang-JMAA-13, M-A-Agarwal-JIA-12, Mansour-Obaid-Ars-12, Mortici-NA-11, Mortici-CAM-10, Mortici-AMC-10, Mortici-AA-10}.
\end{rem}

\section{A best approximation formula}

With the help of Lemma~\ref{m1-Lemma1.1}, we first provide a best approximation formula of Wallis ratio $W_n$.

\begin{thm}\label{best-approx-thm1}
The approximation formula
\begin{equation}\label{w1}
W_{n}\sim \sqrt{\frac{e}{\pi}}\,\biggl(1-\frac{1}{2n}\biggr)^{n}\frac{1}{
\sqrt{n}\,},\quad n\to\infty
\end{equation}
is the best approximation of the form
\begin{equation}\label{a}
W_{n}\sim \sqrt{\frac{e}{\pi}}\,\biggl(1-\frac{1}{2n}\biggr)^{n}\frac{\sqrt{
n+a}}{n},\quad n\to\infty ,
\end{equation}
where $a$ is a real parameter.
\end{thm}

\begin{proof}
Define $z_n(a)$ by
\begin{equation*}
W_{n}=\sqrt{\frac{e}{\pi}}\,\biggl(1-\frac{1}{2n}\biggr)^{n}\frac{\sqrt{n+a}\,}{n}\exp z_{n}(a), \quad n\ge 1.
\end{equation*}
It is not difficult to see that $z_{n}(a)\to0$ as $n\to\infty$, A direct computation gives
\begin{equation*}
z_{n}(a)-z_{n+1}(a)=-\frac{a}{2n^{2}}+\biggl(\frac{1}{2}a+\frac{1}{2}a^{2}+\frac{1}{12}\biggr) \frac{1}{n^{3}} +O\biggl(\frac{1}{n^{4}}\biggr)
\end{equation*}
and
\begin{equation*}
\lim_{n\to\infty}\bigl\{n^2[z_{n}(a)-z_{n+1}(a)]\bigr\}=-\frac{a}2.
\end{equation*}
Making use of Lemma~\ref{m1-Lemma1.1}, we immediately see that the sequence $\{z_{n}(a):n\in\mathbb{N}\}$ converges fastest only when $a=0$. The proof of Theorem~\ref{best-approx-thm1} is complete.
\end{proof}

\begin{rem}
The approximation formula~\eqref{w1} is an improvement of~\eqref{w}, since the approximation formula~\eqref{w} is the special case $a=-1$ in~\eqref{a}.
\end{rem}

\section{An asymptotic series associated to~\eqref{w1}}

In this section, by discovering an asymptotic series and a single-sided inequality for Wallis ratio, we further generalize the approximation formula~\eqref{w1} and improve the left-hand side of the double inequality~\eqref{u}.

\begin{thm}\label{series-approxim-thm}
As $n\to\infty$, we have
\begin{equation*}
W_{n}\sim \sqrt{\frac{e}{\pi}}\,\biggl(1-\frac{1}{2n}\biggr)^{n}\frac{1}{\sqrt{n}\,}
\exp\biggl(\frac{1}{24n^{2}}+\frac{1}{48n^{3}}+\frac{1}{160n^{4}}+\frac{1}{960n^{5}}+\dotsm\biggr).
\end{equation*}
\end{thm}

\begin{proof}
Recall from~\cite{m2} that, to an approximation formula $f(n) \sim g(n) $, the following asymptotic
series is associated
\begin{equation*}
f(n) \sim g(n) \exp\Biggl(\sum_{k=1}^{\infty }\frac{a_{k}}{n^{k}}\Biggr) ,
\end{equation*}
where $a_{k}$ for $k\ge 2$ is a solution of the following infinite triangular system
\begin{equation}\label{s}
a_{1}-\binom{k-1}{1}a_{2}+\dotsm+(-1)^{k}\binom{k-1}{k-2}a_{k-1}=(-1)^{k}x_{k}
\end{equation}
and $x_{k}$ are coefficients of the expansion
\begin{equation*}
\ln \frac{f(n)g(n+1)}{g(n)f(n+1)}=\sum_{k=2}^{\infty }\frac{x_{k}}{n^{k}}.
\end{equation*}
Replacing $f(n)$ and $g(n)$ by $W_{n}$ and $\sqrt{\frac{e}{\pi}}\bigl(1-\frac{1}{2n}\bigr)^{n}\frac{1}{\sqrt{n}\,}$ respectively yields
\begin{equation*}
\ln \frac{f(n)g(n+1)}{g(n)f(n+1)}=\sum_{k=2}^{\infty }(-1)
^{k}\biggl[\frac{1+(-1)^{k}}{(k+1) 2^{k+1}}-\frac{1
}{k+1}+\frac{1}{2k}\biggr]\frac{1}{n^{k}}.
\end{equation*}
Hence, the system~\eqref{s} becomes
\begin{equation*}
a_{1}-\binom{k-1}{1}a_{2}+\dotsm+(-1)^{k}\binom{k-1}{k-2}a_{k-1}=\frac{1+(-1)^{k}}{(k+1) 2^{k+1}}-
\frac{1}{k+1}+\frac{1}{2k}
\end{equation*}
which has a solution
\begin{equation*}
a_{1}=0,\quad a_{2}=\frac{1}{24},\quad a_{3}=\frac{1}{48},\quad a_{4}=\frac{1}{160},\quad a_{5}=\frac{1}{960},\quad\dotsc.
\end{equation*}
The proof of Theorem~\ref{series-approxim-thm} is complete.
\end{proof}

\begin{thm}\label{Q-C-thm2.1}
For every integer $n\ge 1$, we have
\begin{equation}\label{Q-C-ineq2.1}
W_{n}>\sqrt{\frac{e}{\pi}}\,\biggl(1-\frac{1}{2n}\biggr)^{n}\frac{1}{\sqrt{n}\,}
\exp \biggl(\frac{1}{24n^{2}}+\frac{1}{48n^{3}}+\frac{1}{160n^{4}}+\frac{1}{960n^{5}}\biggr).
\end{equation}
\end{thm}

\begin{proof}
It suffices to prove
\begin{equation*}
\alpha_{n}=n\ln \biggl(1-\frac{1}{2n}\biggr) -\frac{1}{2}\ln n-\ln \frac{(2n-1) !!}{(2n)!!}+\ln \sqrt{\frac{e}{\pi}}\,+h(n)<0,
\end{equation*}
where
\begin{equation*}
h(x) =\frac{1}{24x^{2}}+\frac{1}{48x^{3}}+\frac{1}{160x^{4}}+
\frac{1}{960x^{5}}.
\end{equation*}
Because $\alpha_{n}$ converges to $0$, it is sufficient to show that the sequence $\{\alpha_{n}:n\in\mathbb{N}\}$ is strictly increasing. It is not difficult to obtain $\alpha_{n+1}-\alpha_{n}=s(n)$, where
\begin{align*}
s(x)&=(x+1) \ln\biggl(1-\frac{1}{2x+2}\biggr) -x\ln\biggl(1-\frac{1}{2x}\biggr)\\
&\quad-\frac{1}{2}\ln\biggl(1+\frac{1}{x}\biggr)-\ln \frac{2x+1}{2x+2}+h(x+1) -h(x),\\
s''(x)& = \frac{C(x-1) }{32x^{7} (x+1)^{7}(2x+1)^{2}(2x-1)^{2}}\\
&>0,
\end{align*}
and
\begin{align*}
C(x) &=4913+33387x+98177x^{2}+164799x^{3}+174543 x^{4} \\
&\quad+121173x^{5}+55197x^{6}+15920x^{7}+2640x^{8}+192x^{9}.
\end{align*}
Accordingly, the function $s(x)$ is strictly convex on $[1,\infty)$. Combing this with the fact that $\lim_{x\to\infty}s(x)=0$ reveals that the function $s(x)$ on $[1,\infty)$, and so the sequence $\{s(n):n\in\mathbb{N}\}$, is positive. The proof of Theorem~\ref{Q-C-thm2.1} is complete.
\end{proof}

\section{A new approximation formula and a double inequality}

Finally we will find a new approximation formula and a double inequality for Wallis ratio $W_n$.

\begin{thm}\label{m-thm-best-approx}
As $n\to \infty$, we have
\begin{equation}\label{m}
W_{n}\sim \mu _{n}:=\sqrt{\frac{e}{\pi}}\,\biggl[1-\frac{1}{2(n+1/3)}\biggr]^{n+1/3} \frac{1}{\sqrt{n}\,}.
\end{equation}
\end{thm}

\begin{proof}
Motivated by~\eqref{w1}, we now ask for the best approximation of the form
\begin{equation*}
W_{n}\sim \sqrt{\frac{e}{\pi}}\,\biggl[1-\frac{1}{2(n+b)}\biggr]
^{n+c}\frac{1}{\sqrt{n}\,},\quad n\to\infty ,
\end{equation*}
where $b$ and $c$ are real parameters. For this, let
\begin{equation*}
W_{n}=\sqrt{\frac{e}{\pi}}\,\biggl[1-\frac{1}{2(n+b)}\biggr]
^{n+c}\frac{1}{\sqrt{n}\,}\exp \beta_{n}(b,c).
\end{equation*}
Then an easy calculation leads to
\begin{multline*}
\beta_{n}(b,c)-\beta_{n+1}(b,c)=\frac{1}{2}(c-b) \frac{1}{n^{2}} +\biggl(b^{2}-bc-\frac{1}{4}c+\frac{1}{12}\biggr) \frac{1}{n^{3}} \\
+\biggl(\frac{1}{4}c-\frac{1}{8}b+\frac{3}{4}bc-\frac{3}{8}b^{2} -\frac{3}{2}b^{3}+\frac{3}{2}b^{2}c-\frac{1}{16}\biggr) \frac{1}{n^{4}}+O\biggl(\frac{1}{n^{5}}\biggr).
\end{multline*}
This implies that
\begin{equation*}
\lim_{n\to\infty}\bigl\{n^2[\beta_{n}(b,c)-\beta_{n+1}(b,c)]\bigr\}=\frac{c-b}2
\end{equation*}
and
\begin{equation*}
\lim_{n\to\infty}\bigl\{n^2[\beta_{n}(b,b)-\beta_{n+1}(b,b)]\bigr\}=\frac{3b-1}{12}.
\end{equation*}
By Lemma~\ref{m1-Lemma1.1}, it follows that the sequence $\{\beta_{n}(b,c):n\in\mathbb{N}\}$ converges fastest only when $b=c=\frac13$. The proof of Theorem~\ref{m-thm-best-approx} is complete.
\end{proof}

\begin{rem}
We note that the approximation formula~\eqref{m} is the most accurate possible among a class of approximation formulas mentioned above. The numerical computation in Table~\ref{table-numer-1} shows the superiority of~\eqref{m} over~\eqref{w}.
\begin{table}[hbtp]
\caption{Numerical computation}
\begin{tabular}{|c|c|c|}
\hline
$n$ & $W_{n}-\chi _{n}$ & $W_{n}-\mu _{n}$ \\ \hline
$50$ & $8. 0124\times 10^{-4}$ & $4.
4198\times 10^{-9}$ \\ \hline
$100$ & $2. 8269\times 10^{-4}$ & $3.
9124\times 10^{-10}$ \\ \hline
$250$ & $7. 1425\times 10^{-5}$ & $1.
5850\times 10^{-11}$ \\ \hline
$1000$ & $8. 9225\times 10^{-6}$ & $1.
2388\times 10^{-13}$ \\ \hline
\end{tabular}
\label{table-numer-1}
\end{table}
\end{rem}

\begin{thm}\label{Q-C-thm1.1}
For every integer $n\ge 1,$ we have
\begin{multline}
\sqrt{\frac{e}{\pi}}\,\biggl[1-\frac1{2(n+1/3)}\biggr]
^{n+1/3}\frac{1}{\sqrt{n}\,}<W_{n}\\
<\sqrt{\frac{e}{\pi}}\, \biggl[1-\frac1{2(n+1/3)}\biggr]^{n+1/3}\frac{1}{\sqrt{n}\,}
\exp\biggl(\frac{1}{144n^{3}}\biggr).
\end{multline}
\end{thm}

\begin{proof}
It is sufficient to prove
\begin{equation*}
b_{n}=\biggl(n+\frac{1}{3}\biggr) \ln\biggl(1-\frac{1}{2(n+1/3)}\biggr) -\frac{1}{2}\ln n-\ln \frac{(2n-1)!!}{(
2n) !!}+\ln \sqrt{\frac{e}{\pi}}\,<0
\end{equation*}
and
\begin{equation*}
c_{n}=b_{n}+\frac{1}{144n^{3}}>0.
\end{equation*}
Because $b_{n}$ and $c_{n}$ converge to $0$, it suffices to show that $b_{n}$ is strictly increasing and $c_{n}$ is strictly decreasing. For this, we discuss the differences $b_{n+1}-b_{n}=p(n)$ and $c_{n+1}-c_{n}=q(n)$, where
\begin{align*}
p(x)&=\biggl(x+\frac{4}{3}\biggr) \ln\biggl(1-\frac{1}{2(x+4/3)}\biggr) -\biggl(x+\frac{1}{3}\biggr) \ln\biggl(1-\frac{1}{2(x+1/3)}\biggr) \\
&\quad-\frac{1}{2}\ln \biggl(1+\frac{1}{x}\biggr) -\ln \frac{2x+1}{2x+2}
\end{align*}
and
\begin{equation*}
q(x) =p(x) +\frac{1}{144(x+1)^{3}}-
\frac{1}{144x^{3}}.
\end{equation*}
Since
\begin{equation*}
p''(x) =\frac{A(x-1) }{2x^{2}(
3x+1) (3x+4) (x+1)^{2}(2x+1)
^{2}(6x-1)^{2}(6x+5)^{2}}>0
\end{equation*}
and
\begin{equation*}
q''(x) =-\frac{B(x-1) }{12x^{5}(
3x+1) (3x+4) (2x+1)^{2}(6x-1)
^{2}(x+1)^{5}(6x+5)^{2}}<0,
\end{equation*}
where
\begin{align*}
A(x)&=351068+1516131x+2684091x^{2}+2495340x^{3}\\
&\quad+1285956x^{4}+348624x^{5}+38880x^{6}
\end{align*}
and
\begin{align*}
B(x)
&=6780036+50421819x+166596550x^{2}+322415601x^{3}\\
&\quad+405307306x^{4} +346439295x^{5}+204449525x^{6}+82629900x^{7}\\
&\quad+22094730x^{8} +3618864 x^{9}+305208x^{10}+7776x^{11},
\end{align*}
it follows that $p(x)$ is strictly convex and $q(x)$ is strictly concave on $[1,\infty)$. As a result, considering the fact that $\lim_{x\to\infty}p(x)=\lim_{x\to\infty}q(x)=0$, we derive that $p(x)>0$ and $q(x)<0$ on $[1,\infty)$. Consequently, the sequences $\{p(n):n\in\mathbb{N}\}$ and $\{q(n):n\in\mathbb{N}\}$ are positive. The proof of Theorem~\ref{Q-C-thm1.1} is complete.
\end{proof}

\subsection*{Acknowledgements}
The work of the second author was supported in part by the Romanian National Authority for Scientific Research, CNCS-UEFISCDI, under Grant No. PN-II-ID-PCE-2011-3-0087.

\end{document}